\documentclass[a4paper]{article}
\usepackage{amssymb,amsmath,enumerate,amsthm,url}
\usepackage[all,cmtip]{xy}
\usepackage{graphicx}
\usepackage[numbers]{natbib}
\usepackage{bibentry}

\providecommand{\eqnref}[1]{(\ref{e:#1})}
\providecommand{\secref}[1]{\S\ref{s:#1}}
\theoremstyle{plain}
\newtheorem{theorem}{Theorem}[section]
\newtheorem{lemma}[theorem]{Lemma}

\newtheorem{fact}[theorem]{Fact}
\newtheorem*{fact*}{Fact}

\newtheorem{claim}[theorem]{Claim}
\newtheorem*{claim*}{Claim}
\theoremstyle{definition}
\newtheorem{definition}[theorem]{Definition}
\newtheorem*{definition*}{Definition}

\newtheorem*{notation*}{Notation}
\theoremstyle{remark}
\newtheorem{remark}[theorem]{Remark}
\newtheorem*{remark*}{Remark}

\newtheorem*{example*}{Example}
\newtheorem{examples}[theorem]{Examples}
\newtheorem*{examples*}{Examples}

\newtheorem*{note*}{Note}

\newtheorem*{question*}{Question}
\begin{document}
\providecommand{\C}{\mathbb{C}}
\providecommand{\N}{\mathbb{N}}
\providecommand{\R}{\mathbb{R}}
\providecommand{\Z}{\mathbb{Z}}
\providecommand{\Q}{\mathbb{Q}}
\providecommand{\A}{\mathcal{A}}
\providecommand{\B}{\mathcal{B}}
\providecommand{\U}{\mathcal{U}}
\providecommand{\g}{\mathfrak{g}}
\providecommand{\h}{\mathfrak{h}}
\providecommand{\m}{\mathfrak{m}}
\providecommand{\s}{\mathfrak{s}}
\providecommand{\so}{\mathfrak{so}}
\providecommand{\su}{\mathfrak{su}}
\providecommand{\Th}{\operatorname{Th}}
\providecommand{\id}{\operatorname{id}}
\providecommand{\st}{\operatorname{st}}
\providecommand{\mod}{\operatorname{mod}}
\providecommand{\Ad}{\operatorname{Ad}}
\providecommand{\ad}{\operatorname{ad}}
\providecommand{\SO}{\operatorname{SO}}
\providecommand{\Spin}{\operatorname{Spin}}
\providecommand{\SU}{\operatorname{SU}}
\providecommand{\GL}{\operatorname{GL}}
\providecommand{\pr}{\operatorname{pr}}
\providecommand{\Aut}{\operatorname{Aut}}
\providecommand{\End}{\operatorname{End}}
\providecommand{\Mat}{\operatorname{Mat}}
\providecommand{\alg}{\operatorname{alg}}
\providecommand{\ker}{\operatorname{ker}}
\providecommand{\trd}{\operatorname{trd}}
\providecommand{\cl}{\operatorname{cl}}
\providecommand{\dcl}{\operatorname{dcl}}
\providecommand{\tp}{\operatorname{tp}}
\providecommand{\RCF}{\operatorname{RCF}}
\providecommand{\ACF}{\operatorname{ACF}}
\providecommand{\RCVF}{\operatorname{RCVF}}
\providecommand{\ACVF}{\operatorname{ACVF}}

\renewcommand{\O}{\mathcal{O}}

\renewcommand{\o}{\circ}

\renewcommand{\l}{\mathfrak{l}}

\newcommand{\Ru}{\mathcal{R}}
\newcommand{\Goo}{{G^{00}}}
\newcommand{\Hoo}{{H^{00}}}
\newcommand{\Soo}{{S^{00}}}
\newcommand{\gm}{\g_\m}

\newcommand{\lie}[1]{#1}

\newcommand{\lieG}{\lie{G}}
\newcommand{\lieH}{\lie{H}}
\newcommand{\lieS}{\lie{S}}
\newcommand{\lieT}{\lie{T}}

\newcommand{\inner}[1]{\left<#1\right>}

\newcommand{\conj}{{ }^\wedge}

\newcommand{\recalldefn}[1]{\emph{#1}}
\newcommand{\defn}[1]{{\bf #1}}

\newcommand{\terpAr}{\ar@{~>}}
\newcommand{\biterpAr}[2]{\terpAr@<0.5ex>[#1]^{#2}}
\newcommand{\terp}[2]{\xymatrix{#1 \terpAr[r] & #2}}

\newcommand{\CHECKME}{\textbf{CHECKME}}
\newcommand{\FIXME}[1]{\textbf{FIXME: #1}}
\newcommand{\TODO}[1]{\textbf{TODO: #1}}

\title{Definability in the group of infinitesimals of a compact Lie group}
\author{M. Bays \& Y. Peterzil}

\maketitle

\abstract{
  We show that for $G$ a simple compact Lie group, the infinitesimal subgroup 
  $\Goo$ is bi-interpretable with a real closed convexly valued field.
  We deduce that for $G$ an infinite definably compact group definable in an 
  o-minimal expansion of a field, $\Goo$ is bi-interpretable with the disjoint 
  union of a (possibly trivial) $\Q$-vector space and finitely many (possibly 
  zero) real closed valued fields.
  We also describe the isomorphisms between such infinitesimal subgroups, and 
  along the way prove that every {\em definable} field in a real closed 
  convexly valued field $R$ is definably isomorphic to $R$.
}
\newcommand{\sub}{\subseteq}

\section{Introduction}
\label{s:intro}

Let $G$ be compact linear Lie group, by which we mean a compact closed Lie 
subgroup of $\GL_d(\R)$ for some $d \in \N$, e.g.\ $G=SO_d(\R)$.
By Fact~\ref{f:chevalley}(i) below, any compact Lie group is isomorphic to a linear 
Lie group.

Let $\Ru \succ  \R$ be a real closed field properly extending the real field.
By Fact~\ref{f:chevalley}(ii), $G$ is the group of $\R$-points of an algebraic 
subgroup of $\GL_d$ over $\R$, and we write $G(\Ru) \leq  \GL_d(\Ru)$ for the 
$\Ru$-points of this algebraic group.

Let $\st : \O \rightarrow  \R$ be the standard part map,
the domain $\O = \bigcup_{n \in \N} [-n,n] \subseteq  \Ru$ of which is a valuation ring 
in $\Ru$.
Let
\[ \m = \st^{-1}(0) = \bigcap_{n \in \N_{>0}} [{ -\frac1n,\frac1n }] \subseteq  \Ru \]
be the maximal ideal of $\O$.
A field equipped with a valuation ring, such as $(\Ru;+,\cdot,\O)$, is known 
as a {\em valued} {\em field}. The complete first-order theory of $(\Ru;+,\cdot,\O)$ 
is the theory $\RCVF$ of non-trivially convexly valued real closed fields 
\cite{cherlinDickmann}.

Since $G$ is compact, $\st : \O \rightarrow  \R$ induces a totally defined homomorphism 
$\st_G : G(\Ru) \twoheadrightarrow  G$.
The kernel $\ker(\st_G) \triangleleft G$ is the ``infinitesimal subgroup'' of $G$ in 
$\Ru$. Our main results below describe definability in this group.

In fact, by Fact~\ref{f:DCLit}(iv) below, $\ker(\st_G)$ is precisely $\Goo(\Ru)$, 
the set of $\Ru$-points of the smallest bounded-index $\bigwedge$-definable 
subgroup $\Goo$ of the semialgebraic group $G$.
In terms of matrices,
\begin{equation} \label{e:Goo}
  \Goo(\Ru) = G(\Ru) \cap (I+\Mat_d(\m)) = \bigcap_{n \in \N_{>0}} ( G(\Ru) 
  \cap (I + \Mat_d([-\frac1n,\frac1n]) )
\end{equation}
(where we write $\Mat_d(X)$ for the set of matrices with entries in a set $X 
\subseteq  \Ru$). Note in particular that $\Goo(\Ru)$ is definable in the valued field 
$(\Ru;+,\cdot,\O)$.

We adopt the convention that ``definable'' always means definable with 
parameters: a {\em definable} {\em set} in a structure $\A=(A;\ldots )$ is a subset of a 
Cartesian power $A^n$ defined by a first-order formula in the language of $\A$ 
with parameters from $A$.

Recall that an {\em interpretation} of a (one-sorted) structure $\A$ in a 
structure $\B$ is a bijection of the universe of $\A$ with a definable set
in $\B$, or more generally with the quotient of a definable set by a definable 
equivalence relation, such that the image of any definable set in $\A$ is 
definable in $\B$. A {\em definition} of $\A$ in $\B$ is an interpretation whose 
codomain is a definable set rather than a definable quotient; in fact, these 
are the only interpretations which will arise in the main results of this 
article. There is an obvious notion of composition of interpretations.
A pair of interpretations of $\A$ in $\B$ and of $\B$ in $\A$ form a 
{\em bi-interpretation} if the composed interpretations of $\A$ in $\A$ and of 
$\B$ in $\B$ are definable maps in $\A$ and $\B$ respectively.
If $\A$ and $\B$ are bi-interpretable,
then the structure induced by $\B$ on the image of $\A$ is precisely the 
structure of $\A$,
i.e.\ the definable sets are the same whether viewed in $\A$ or in $\B$.
Indeed, if $(f,g)$ is a bi-interpretation, then given a subset $X \subseteq  \A^n$, if 
$X$ is $\A$-definable then $f(X)$ is $\B$-definable, and conversely if $f(X)$ 
is $\B$-definable then $g(f(X))$ and hence $X$ is $\A$-definable.

We adopt the following terminology throughout the paper.
A {\em linear} Lie group is a closed Lie subgroup of $\GL_d(\R)$ for some $d \in 
\N$.
A Lie group is {\em simple} if it is connected and its Lie algebra is simple; a 
Lie algebra is simple if it is non-abelian and has no proper non-trivial 
ideal.
(A simple Lie group may have non-trivial discrete centre; however, the 
underlying abstract group of a {\em centreless} simple Lie group is simple.)
Our first main result describes definability in $\Goo(\Ru)$ for $G$ simple:

\begin{theorem} \label{t:mainBiterp}
  \label{T:MAINBITERP}
  Let $G$ be a simple compact linear Lie group.

  Let $\Ru \succ  \R$ be a proper real closed field extension of $\R$.

  Then the inclusion $\iota : \Goo(\Ru) \ensuremath{\lhook\joinrel\relbar\joinrel\rightarrow} G(\Ru)$, viewed as a definition 
  of the group $(\Goo(\Ru);*)$ in the valued field $(\Ru;+,\cdot,\O)$, can be 
  extended to a bi-interpretation: there is a definition $\theta$ of the 
  valued field $(\Ru;+,\cdot,\O)$ in the group $(\Goo(\Ru);*)$ such that the 
  pair $(\iota,\theta)$ form a bi-interpretation.

  In particular, the $(\Goo(\Ru);*)$-definable subsets of powers 
  $(\Goo(\Ru))^n$ are precisely the $(\Ru;+,\cdot,\O)$-definable subsets.
\end{theorem}

\begin{remark}
  The bi-interpretation of Theorem~\ref{t:mainBiterp} requires parameters. Indeed, 
  $\Goo(\Ru)$ has non-trivial inner automorphisms (this follows from 
  Lemma~\ref{l:Coo} below), which under a hypothetical parameter-free 
  bi-interpretation would induce definable non-trivial automorphisms of 
  $(\Ru;+,\cdot,\O)$.

  However, no such automorphism exists. We sketch a proof of this, following a 
  method suggested by Martin Hils. By Fact~\ref{f:semialgebraic}(1),
  if $\sigma$ is an $(\Ru;+,\cdot,\O)$-definable automorphism, then it agrees 
  on some infinite interval $I$ with an $(\Ru;+,\cdot)$-definable map $f$.
  But then for $a,b,c,d \in I$ with $c\neq d$,
  $\sigma(\frac{a-b}{c-d}) = \frac{\sigma(a)-\sigma(b)}{\sigma(c)-\sigma(d)} = 
  \frac{f(a)-f(b)}{f(c)-f(d)}$,
  so $\sigma$ is an $(\Ru;+,\cdot)$-definable field automorphism, thus 
  $\sigma=\id$.
\end{remark}

\begin{remark}
  For $\lieG$ a simple centreless compact Lie group, Nesin and Pillay 
  \cite{NP-compactLie} showed that the group itself, $(G;*)$, is 
  bi-interpretable with a real closed field. They interpret the field by 
  finding a copy of $(\SO_3;*)$ and using the geometry of its involutions. A 
  similar project is carried out in \cite{PPS-simpleAlgGrpsRCF} for definably 
  simple and semisimple groups in o-minimal structures. In the case of $\Goo$, 
  we also find a field by first finding a copy of $\SO_3^{00}$, but the 
  ``global'' approach of considering involutions is not available; in fact 
  $\Goo$ is torsion-free. Instead, we work ``locally'', and obtain the field 
  by applying the o-minimal trichotomy theorem to a definable interval on a 
  curve within $\SO_3^{00}$. This kind of local approach was previously 
  mentioned in an ``added in proof'' remark at the end of 
  \cite{PPS-simpleAlgGrpsRCF} as an alternative method for the case of 
  $(G;*)$, but we have to take care to ensure that structure we apply 
  trichotomy to is definable both o-minimally and in $(\Goo;*)$.
\end{remark}
\begin{remark} \label{r:otherValns}
  One might also consider ``smaller'' infinitesimal neighbourhoods 
  corresponding to larger valuation rings $\O' \supsetneq \O$; 
  Theorem~\ref{t:mainBiterp} holds for these too. More generally, if 
  $(\Ru';+,\cdot,\O') \vDash  \RCVF$ extends $\R$ as an ordered field, then the 
  group $G(\Ru) \cap \Mat_n(\mu')$ is bi-interpretable as in 
  Theorem~\ref{t:mainBiterp} with $(\Ru';+,\cdot,\O')$. Indeed, the existence of 
  suitable parameters for the bi-interpretation is expressed by a sentence in 
  $\RCVF$ with parameters in $\R$, and since $\RCVF$ is complete we can apply 
  Theorem~\ref{t:mainBiterp} to deduce the result.
\end{remark}

We prove Theorem~\ref{t:mainBiterp} in \secref{simple}.

In \secref{borel-tits}, we deduce in the spirit of Borel-Tits a 
characterisation of the group isomorphisms of groups of the form $\Goo$, 
decomposing them as compositions of valued field isomorphisms and isomorphisms 
induced by isomorphisms of Lie groups. In particular, this shows that simple 
compact $G_1$ and $G_2$ have isomorphic infinitesimal subgroups if and only if 
they have isomorphic Lie algebras. The key technical tool is 
Theorem~\ref{t:defbleFieldsRCVF}, which shows that there are no unexpected fields 
definable in $\RCVF = \Th(\Ru;+,\cdot,\O)$.

In \secref{DC}, we generalise Theorem~\ref{t:mainBiterp} to the setting of a 
definably compact group $G$ definable in an o-minimal expansion of a field. 
Here, to say that $G$ is \recalldefn{definably compact} means that any 
definable continuous map $[0,1) \rightarrow  G$ can be completed to a continuous map 
$[0,1] \rightarrow  G$; we refer to \cite{PS-defbleCompactness} for further details on 
this notion.

Define the {\bf disjoint} {\bf union} of 1-sorted structures $M_i$ to be the structure 
$(M_i)_i$ consisting of a sort for each $M_i$ equipped with its own structure, 
with no further structure between the sorts.

\begin{theorem} \label{t:DCBiterp}
  Let $(G;*)$ be an infinite definably compact group definable in a 
  sufficiently saturated o-minimal expansion $M$ of a field.
  Then $(\Goo(M);*)$ is bi-interpretable with the disjoint union of a 
  (possibly trivial) divisible torsion-free abelian group and finitely many 
  (possibly zero) real closed convexly valued fields.
\end{theorem}

To indicate why this is the correct statement, let us note that it can not be 
strengthened to bi-interpretability with a single real closed valued field as 
in Theorem~\ref{t:mainBiterp}: one reason is that $G$ could be commutative, and then 
$(\Goo;*)$ is just a divisible torsion free abelian group and thus does not 
even define a field; another reason is that groups are orthogonal in their 
direct product, so e.g.\ if $G = H\times H$ for a semialgebraic compact group 
$H$ then, viewing $\Goo$ as a definable set in a valued field 
$(\Ru;+,\cdot,\O)$ as above, the diagonal subgroup of $\Goo(\Ru) = 
\Hoo(\Ru)\times\Hoo(\Ru)$ is $(\Ru;+,\cdot,\O)$-definable but not 
$(\Goo(\Ru);*)$-definable.

Note that Theorem~\ref{t:DCBiterp} applies in particular to an arbitrary compact 
linear Lie group, since by Fact~\ref{f:chevalley}(ii) any such group is definable in 
the real field, and is definably compact.

\subsection{Acknowledgements}

We would like to thank Mohammed Bardestani, Alessandro Berarducci, Emmanuel 
Breuillard, Itay Kaplan, and Martin Hils for useful discussion.
We would also like to thank the Institut Henri Poincar\'e and the organisers of 
the trimester ``Model theory, combinatorics and valued fields'', where some of 
this work was done.
We also thank the anonymous referee for suggestions which improved the 
presentation of this paper substantially.
Greg Cherlin and Ali Nesin, on different occasions, have asked about the 
abstract group structure of the infinitesimal subgroup of a compact simple Lie 
group. This article can be seen as a partial answer to their questions.
Finally, we would like to mention the following question of Itay Kaplan which 
initiated the discussions which led to our results (even though our results have in 
the end almost nothing to do with the question): is any $\bigwedge$-definable field 
in an NIP theory itself NIP as a pure field?


\section{Preliminaries}
\label{s:prelims}
\subsection{Notation}
We consider $\Ru$ and $\R$ as fields and $\Goo(\Ru)$ as a group, thus 
``$\Goo(\Ru)$-definable'' means definable in the pure group $(\Goo(\Ru);*)$, 
and ``$\Ru$-definable'' means definable in the field $(\Ru;+,\cdot)$, while we 
write ``$(\Ru;\O)$-definable'' for definability in the valued field 
$(\Ru;+,\cdot,\O)$.

We use exponential notation for group conjugation, $g^h := hgh^{-1}$.
We write group commutators as $(a,b) := aba^{-1}b^{-1}$, reserving $[X,Y]$ for 
the Lie bracket.
For $G$ a group, we write $(G,G)_1$ for the set of commutators in $G$, 
$(G,G)_1 := \{ (g,h) : g,h \in G \}$,
and we write $(G,G)$ for the commutator subgroup, the subgroup generated by 
$(G,G)_1$.

For $G$ a group and $A$ a subset, we write the centraliser of $A$ in $G$ as 
$C_G(A) = \{ g \in G : \forall a \in A.\;(g,a) = e \}$,
and we write $Z(G)$ for the centre $Z(G)=C_G(G)$.
Similarly for $\g$ a Lie algebra and $A$ a subset, we write
the centraliser of $A$ in $\g$ as $C_\g(A) = \{ X \in \g : \forall Y \in A.\; 
[X,Y] = 0 \}$.
We write $C_G(g)$ for $C_G(\{g\})$ and $C_\g(X)$ for $C_\g(\{X\})$.

\subsection{Compact Lie groups}
Proofs of the statements in the following Fact can be found in \cite{OV} as 
Theorem~5.2.10 and Theorem~3.4.5 respectively.
\begin{fact}[Chevalley] \label{f:chevalley}
  \begin{enumerate}[(i)]\item Any compact Lie group $\lieG$ is linear; that is, $\lieG$ is 
  isomorphic to a Lie subgroup of $\GL_d(\R)$ for some $d$.
  \item Compact linear groups are algebraic; that is, any compact subgroup of 
  any $\GL_d(\R)$ is of the form $G(\R)$ for some algebraic subgroup $G \leq  
  \GL_d$ over $\R$.
  \end{enumerate}
\end{fact}

\begin{lemma} \label{l:Coo}
  Suppose $H$ is a connected closed Lie subgroup of a compact linear Lie group 
  $G$.

  Then $C_{G(\Ru)}(H(\Ru)) = C_{G(\Ru)}(H^{00}(\Ru))$.
\end{lemma}
\begin{proof}
  Since $G$ and $H$ are algebraic by Fact~\ref{f:chevalley}(ii), the conclusion can 
  be expressed as a first-order sentence in the complete theory $\RCVF_\R$ of 
  a non-trivially convexly valued real closed field extension 
  $(\Ru;+,\cdot,\O)$ of the trivially valued field $\R$, so we may assume 
  without loss that $\Ru$ is $\aleph_1$-saturated.

  Suppose $x \in C_{G(\Ru)}(H^{00}(\Ru))$.
  It follows from $\aleph_1$-saturation of $\Ru$ that $x \in 
  C_{G(\Ru)}(U(\Ru))$ for some $\R$-definable neighbourhood of the identity $U 
  \subseteq  H$, since $H^{00}(\Ru)$ is the intersection of such (e.g.\ as in 
  \eqnref{Goo}).
  Now we could argue from general results on o-minimal groups (see 
  \cite[Lemma~2.11]{P-groupsFieldsOMin}) and connectedness that 
  $C_{G(\Ru)}(U(\Ru)) = C_{G(\Ru)}(H(\Ru))$, but we can also argue directly as 
  follows.
  $H$ is compact and connected, thus is generated in finitely many steps from 
  $U$. Since $\Ru$ is an elementary extension of $\R$, it follows that 
  $H(\Ru)$ is generated by $U(\Ru)$.
  Hence $x \in C_{G(\Ru)}(H(\Ru))$.
\end{proof}

\subsection{$\SO(3)$}
We recall some elementary facts about the group of spatial rotations 
$\SO(3)=\SO_3(\R)$, its universal cover $\Spin(3)$, and their common Lie 
algebra $\so(3)$, as discussed in e.g.\ \cite[Chapter~6]{woit}.

Any rotation $g \in \SO(3)$ can be completely described as a planar rotation 
$\alpha \in \SO(2)$ around an axis $L$, where $L$ is a ray from the origin in 
$\R^3$. We can identify such a ray $L$ with the unique element of the unit 
sphere $S^2$ which lies on the ray. Writing $\rho : \SO(2) \times S^2 \twoheadrightarrow  
\SO(3)$ for the corresponding map, the only ambiguities in this description 
are that $\rho(\alpha,L) = \rho(-\alpha,-L)$, and the trivial rotation is $e = 
\rho(0,L)$ for any $L$. The centraliser in $\SO(3)$ of a non-trivial 
non-involutary rotation $g=\rho(\alpha,L)$, $2\alpha \neq  0$, is the subgroup 
$C_{\SO(3)}(g) = \rho(\SO(2),L) \cong  \SO(2)$ of rotations around $L$. The 
conjugation action of $\SO(3)$ on itself is by rotation of the axis: 
$\rho(\alpha,L)^g = \rho(\alpha,gL)$, where $gL$ is the image of $L$ under the 
canonical (matrix) left action of $g$ on $\R^3$. The map $\rho$ is continuous 
and $\R$-definable.

This description transfers to non-standard rotations: given $\Ru \succ  \R$,
$\rho$ extends to a map $\rho : \SO_2(\Ru) \times S^2(\Ru) \twoheadrightarrow  \SO_3(\Ru)$.
The infinitesimal rotations are then the rotations about any (non-standard) 
axis by an infinitesimal angle;
i.e.\ $\rho$ restricts to a map $\SO_2^{00}(\Ru) \times S^2(\Ru) \twoheadrightarrow  
\SO_3^{00}(\Ru)$.

The universal group cover of $\SO(3)$ is denoted $\Spin(3)$; the corresponding 
continuous covering homomorphism $\pi : \Spin(3) \rightarrow  \SO(3)$ is a local 
isomorphism with kernel $Z(\Spin(3)) \cong  \pi_1(\SO(3)) \cong  \Z/2\Z$. Since 
$\Spin(3)$ is compact, by Fact~\ref{f:chevalley}(i) we may represent it as $\Spin(3) 
= \Spin_3(\R)$ where $\Spin_3$ is a linear Lie group. Since $\pi$ is an 
$\R$-definable local isomorphism, it induces an isomorphism of Lie algebras 
and an isomorphism of infinitesimal subgroups $\Spin_3^{00}(\Ru) \cong  
\SO_3^{00}(\Ru)$, with respect to which the action by conjugation of $g \in 
\Spin_3(\Ru)$ on $\Spin_3^{00}(\Ru)$ agrees with the action by conjugation of 
$\pi(g)$ on $\SO_3^{00}(\Ru)$.

The Lie algebra $\so(3) \cong  \su(2)$ has $\R$-basis $\{H,U,V\}$ and bracket 
relations
\begin{equation} \label{e:so3}
  [U,V]=H,\; [H,U]=V,\; [V,H]=U .
\end{equation}

The adjoint action $\Ad$ of $g \in \SO(3)$ on $\so(3)$ is by the left matrix 
action with respect to this basis, $\Ad_g(X) = gX$, and similarly for $g \in 
\Spin(3)$ the adjoint representation is $\pi$, i.e.\ $\Ad_g(X) = \pi(g)X$. 

\subsection{Lie theory in o-minimal structures}
\label{s:nsLie}
We recall briefly the Lie theory of a group $G$ definable in an o-minimal 
expansion of a real closed field $R$; see \cite{PPS-defblySimple} for further 
details, but in fact we apply it only to linear algebraic groups, for which it 
agrees with the usual theory for such groups. The Lie algebra of $G$ is the 
tangent space at the identity $\g = L(G) = T_e(G)$, a finite dimensional 
$R$-vector space. For $h \in G(R)$, define $\Ad_h : \g \rightarrow  \g$ to be the 
differential of conjugation by $h$ at the identity, $\Ad_h := d_e(\cdot^h)$, 
and define $\ad : \g \rightarrow  \End_R(\g)$ as the differential of $\Ad : G \rightarrow  
\Aut_R(\g)$ at the identity, $\ad := d_e(\Ad)$. Then the Lie bracket of $\g$ 
is defined as $[X,Y] := \ad_X(Y)$.

The statements above about the adjoint action of $\SO_3$ and $\Spin_3$ on 
$\so_3$ transfer to $\Ru$: for $X \in \so_3(\Ru)$, we have that $\Ad_g(X) = gX$ 
for $g \in \SO_3(\Ru)$, and $\Ad_g(X) = \pi(g)X$ for $g \in \Spin_3(\Ru)$.

\section{Proof of Theorem~\ref{t:mainBiterp}}
\label{s:simple}

Let $G$ and $\Ru$ be as in Theorem~\ref{t:mainBiterp}, namely $G$ is a simple 
compact linear Lie group and $\Ru$ is a proper real closed field 
extension\footnote{
Readers familiar with model theory might be made more comfortable by an 
assumption that $\Ru$ is (sufficiently) saturated. They may in fact freely 
assume this, since the conclusion of the theorem can be seen to not depend on 
the choice of $\Ru$.}
of $\R$.

%
%

In this section, we write $\Goo$ for the group $\Goo(\Ru)$.

We first give an overview of the proof. We begin in \secref{SO300} by finding a 
copy of $\SO(3)$ or $\Spin(3)$ in $G$ which is definable in such a way that 
its infinitesimal subgroup is $\Goo$-definable. A reader who is interested 
already in the case $G=\SO(3)$ may prefer to skip that section on a first 
reading. In \secref{interval} we use the structure of $\SO_3^{00}$ to find in 
it an interval on a centraliser which is definable both in the group and the 
field. In \secref{trich} we see that the non-abelianity of $G$ endows this 
interval with a rich enough structure to trigger the existence of a field by 
the o-minimal trichotomy theorem. Finally, in \secref{Goo}, we use an adjoint 
embedding to see the valuation on this field and obtain bi-interpretability.

\subsection{Finding an $\SO_3^{00}$}
\label{s:SO300}
In this subsection, we find a copy of $\so(3)$ in the Lie algebra of $G$ which 
is the Lie algebra of a Lie subgroup $S \leq  G$, and which moreover is defined 
in such a way that the infinitesimal subgroup $\Soo \leq  \Goo$ is 
$\Goo$-definable.

Let $\g_0 := L(\lieG)$ be the Lie algebra of $\lieG$.

\begin{lemma} \label{l:so3}
  There exist Lie subalgebras $\s \leq  \s' \leq  \g_0$ such that
  \begin{enumerate}[(i)]\item $\s' = C_{\g_0}(C_{\g_0}(\s'))$
  \item $\s = [\s',\s']$
  \item $\s \cong  \so(3)$.
  \end{enumerate}
\end{lemma}
\begin{proof}
  In this proof, and in this proof alone, we assume familiarity with the basic 
  theory and terminology of the root space decomposition of a semisimple Lie 
  algebra. This can be found in e.g.\ \cite[\S\S II.4, II.5]{Knapp}.


  We write $V^*$ for the dual space of a vector space $V$.

  Let $\g := \g_0 \otimes_\R \C$ be the complexification of $\g_0$, and let
  $\overline{\cdot} : \g \rightarrow  \g$ be the corresponding complex conjugation.

  Let $\h_0 \leq  \g_0$ be a Cartan subalgebra of $\g_0$, meaning that the 
  complexification $\h := \h_0 \otimes_\R \C$ is a Cartan subalgebra of $\g$; we 
  can take $\h_0 := L(\lieT)$ where $\lieT$ is a maximal torus of $\lieG$ 
  (where a torus of $\lieG$ is a Lie subgroup isomorphic to a power of the 
  circle group).

  Now since $\lieG$ is simple, $\g$ is semisimple and so admits a root space 
  decomposition $\g = \h \oplus \bigoplus_{\alpha \in \Delta} \g_\alpha$ where each root 
  space $\g_\alpha$ is the 1-dimensional eigenspace of $\ad(\h)$ with 
  eigenvalue $\alpha \in \h^* \setminus \{0\}$, i.e.\ if $H \in \h$ and $X \in \g_\alpha$ 
  then $[H,X] = \alpha(H)X$. The roots $\alpha$ span $\h^*$. If $\alpha \in 
  \Delta$ then $\C\alpha \cap \Delta = \{\alpha,-\alpha\}$.
  We have $[\g_\alpha,\g_{-\alpha}] \leq  \h$.

  Since $\lieG$ is compact, each root $\alpha \in \Delta$ takes purely 
  imaginary values on $\h_0$, thus $i\alpha\negmedspace\restriction_{\h_0} \in \h_0^*$, and the 
  subalgebra $\s_{\alpha}$ of $\g_0$ generated by $\l_\alpha := \g_0 \cap 
  (\g_{\alpha} \oplus \g_{-\alpha})$ is isomorphic to $\so(3)$ (see 
  \cite[Proposition~26.4]{FultonHarris}, or \cite[(4.61)]{Knapp}).

  Explicitly, if $E_\alpha \in \g_\alpha \setminus \{0\}$, so $\overline{E_\alpha} \in 
  \g_{-\alpha}$, then $\l_\alpha$ is spanned by $U_\alpha := iE_\alpha - 
  i\overline{E_\alpha}$ and $V_\alpha := E_\alpha + \overline{E_\alpha}$.
  Then for $H \in \h_0$ we have
  \begin{equation} \label{e:adl}
    [H,U_\alpha] = i\alpha(H) V_\alpha \;\;\text{and}\;\; [H,V_\alpha] = 
    -i\alpha(H) U_\alpha ,
  \end{equation}
  and $[U_\alpha,V_\alpha] = 2i[E_\alpha,\overline{E_\alpha}] \in \s_\alpha 
  \cap \h_0$. Using that $\g_0$ admits an $\Ad(\lieG)$-invariant inner 
  product, one can argue (see the proof of \cite[(4.56)]{Knapp} for details) 
  that $\alpha([E_\alpha,\overline{E_\alpha}]) < 0$, and hence that after 
  renormalising $E_\alpha$, the $\R$-basis
  $\{[U_\alpha,V_\alpha], U_\alpha, V_\alpha\}$
  for $\s_\alpha$ satisfies the standard bracket relations \eqnref{so3} of 
  $\so(3)$.

  Let $\alpha_0 \in \Delta$.
  Let $\s := \s_{\alpha_0}$ and $\l := \l_{\alpha_0}$,
  and let $\s' := \h_0 \oplus \l$.
  It follows from the bracket relations above that $\s'$ is a subalgebra of 
  $\g_0$, and the commutator subalgebra $[\s',\s']$ is precisely $\s$.

  It remains to show that $\s' = C_{\g_0}(C_{\g_0}(\s'))$.

  It follows from \eqnref{adl} and $C_{\g_0}(\h_0)=\h_0$ that
  $C_{\g_0}(\s') = \h_0 \cap C_{\g_0}(\l) = \ker(\alpha_0\negmedspace\restriction_{\h_0})$.

  Now for $\alpha' \in \Delta$ and $W \in \g_{\alpha'} \setminus \{0\}$, we have 
  $[W,\ker(\alpha_0\negmedspace\restriction_{\h_0})] = 0$ if and only if $\ker(\alpha'\negmedspace\restriction_{\h_0}) \supseteq 
  \ker(\alpha_0\negmedspace\restriction_{\h_0})$;
  since $i\alpha'\negmedspace\restriction_{\h_0},i\alpha_0\negmedspace\restriction_{\h_0} \in \h_0^*$, this holds if and 
  only if $\alpha' \in \R\alpha_0 \cap \Delta = \{\alpha_0,-\alpha_0\}$.
  Thus
  $C_{\g_0}(C_{\g_0}(\s')) = \g_0 \cap (\h + \g_{\alpha_0} + \g_{-\alpha_0}) = 
  \s'$, as required.
\end{proof}

  Before the next lemma we recall some facts and terminology from Lie theory.
  An \recalldefn{integral subgroup} of $\lieG$
  is a connected Lie group which is an abstract subgroup of $\lieG$ via an
  inclusion map which is an immersion. An integral subgroup is a Lie subgroup
  iff it is closed in $\lieG$ (see
  \cite[Proposition~III.6.2.2]{Bourbaki-Lie}). The map $\lieH \mapsto  L(\lieH)$ of
  taking the Lie algebra is a bijection between integral subgroups of $\lieG$
  and Lie subalgebras of $\g_0=L(\lieG)$ (see
  \cite[Theorem~III.6.2.2]{Bourbaki-Lie}).

\begin{lemma} \label{l:spin3}
  There exists a closed Lie subgroup $S \leq  \lieG$ isomorphic to either 
  $\Spin(3)$ or $\SO(3)$, such that the infinitesimal subgroup $\Soo = S(\Ru) 
  \cap \Goo \cong  \SO_3^{00}$ is $\Goo$-definable as a subgroup of $\Goo$.
\end{lemma}
\begin{proof}
  Let $\s$ and $\s'$ be as in Lemma~\ref{l:so3}.
  Let $\lieS$ and $\lieS'$ be the integral subgroups of $\lieG$ with Lie 
  algebras $\s$ and $\s'$, respectively.
  Since $\s' = C_{\g_0}(C_{\g_0}(\s'))$, we have by 
  \cite[Proposition~III.9.3.3]{Bourbaki-Lie} that $\lieS' = 
  C_{\lieG}(C_{\lieG}(\lieS'))$.
  In particular, $\lieS'$ is closed.

  Since $[\s',\s']=\s$ and $\s$ is an ideal in $\s'$, we have by 
  \cite[Proposition~III.9.2.4]{Bourbaki-Lie} that the commutator subgroup 
  $(\lieS',\lieS')$ is equal to $\lieS$.
  Now $\s$ is isomorphic to $\so(3)$ and $\Spin(3)$ is simply connected, thus 
  (by \cite[Theorem~III.6.3.3]{Bourbaki-Lie}) $\lieS$ is the image of a 
  non-singular homomorphism $\Spin(3) \rightarrow  \lieG$ with central kernel,
  and so since $\Spin(3)$ is compact and its centre is of order 2, $\lieS$ is 
  a closed Lie subgroup of $\lieG$ isomorphic either to $\Spin(3)$ or to 
  $\SO(3) \cong  \Spin(3) / Z(\Spin(3))$.



  Now $S$ is a closed Lie subgroup of $G$, and thus is a compact linear Lie 
  group, and its infinitesimal subgroup $\Soo$ is correspondingly the subgroup 
  $S(\Ru) \cap \Goo$ of $\Goo$. The same goes for the closed Lie subgroups 
  $S'$ and $S'' := C_G(S')$ of $G$.

  \begin{claim}
    $(S')^{00} = C_{\Goo}(C_{\Goo}(S'(\Ru)))$.
  \end{claim}
  \begin{proof}
    \begin{align*} (S')^{00}
    &= S'(\Ru) \cap \Goo \\
    &= C_{G(\Ru)}(C_{G(\Ru)}(S'(\Ru))) \cap \Goo \\
    &= C_{\Goo}(C_{G(\Ru)}(S'(\Ru))) \\
    &= C_{\Goo}(S''(\Ru)) \\
    &= C_{\Goo}((S'')^{00}) & \text{(by Lemma~\ref{l:Coo})} \\
    &= C_{\Goo}(C_{G(\Ru)}(S'(\Ru)) \cap \Goo) \\
    &= C_{\Goo}(C_{\Goo}(S'(\Ru))) .\end{align*}
  \end{proof}

  By the Descending Chain Condition for definable groups in o-minimal 
  structures, \cite[Corollary~1.16]{PPS-defblySimple}, there exists a finite 
  set $X_0\subseteq C_{G^{00}}(S')\subseteq G^{00}$ such that 
  $C_G(C_{G^{00}}(S'))=C_G(X_0)$. Thus,
  $$S^{00}=C_G(X_0)\cap G^{00}=C_{G^{00}}(X_0)$$ is $\Goo$-definable.


  Meanwhile, for the commutator subgroup, we have
  $((S')^{00},(S')^{00}) \subseteq  (S',S')^{00} = S^{00}$.
  Now $(\SO_3^{00},\SO_3^{00})_1 = \SO_3^{00}$;
  in fact \cite{dAndreaMaffei} proves $(\Goo,\Goo)_1=\Goo$ for any compact 
  semisimple Lie group, but one can also see it directly in this case,
  since $(\SO_3^{00},\SO_3^{00})_1$ is invariant under conjugation by 
  $\SO_3(\Ru)$ and can be seen to contain infinitesimal rotations of all 
  infinitesimal angles.
  So since $(S')^{00} \supseteq S^{00} \cong  \SO_3^{00}$,
  we have $((S')^{00},(S')^{00})_1=S^{00}$.
  Hence $\Soo$ is $\Goo$-definable.

  Thus $S$ is as required.
\end{proof}

\subsection{Finding a group interval in $\Soo$}
\label{s:interval}
Let $S \leq  G$ be as given by Lemma~\ref{l:spin3}. Now $S$ is isomorphic to $\SO_3$ or 
$\Spin_3$ via an isomorphism which has compact graph and hence is 
$\R$-definable, so let the $\R$-definable map $\pi : S \twoheadrightarrow  \SO_3$ be this 
isomorphism or its composition with the universal covering homomorphism 
respectively. Then $\pi$ induces an isomorphism $\pi : S^{00} \xrightarrow{\cong} 
\SO_3^{00}$.

From now on, we work in $\Ru$, and consider $S$, $G$, $\SO_2$, $\SO_3$, and 
$\Spin_3$ as $\R$-definable (linear algebraic) groups rather than as Lie 
groups; moreover, {\bf we write $S$ for $S(\Ru)$}, and similarly with $G$, 
$\SO_2$, $\SO_3$, and $\Spin_3$.

Let $g \in \Soo \setminus \{e\}$. Then $C_S(g) = \pi^{-1}(C_{\SO_3}(\pi(g)))$, and 
$C_{\SO_3}(\pi(g)) \cong  \SO_2$ since $g$ is not torsion.
Thus the circular order on $\SO_2$ induces\footnote{
Explicitly, say $\pi(g) = \rho(\alpha,L)$; then if $\pi_{(1,2)} : \SO_2 \rightarrow  
\Ru$ extracts the top-right matrix element,
then
$x_1 < x_2 \Leftrightarrow  \pi_{(1,2)}(\beta_1) < \pi_{(1,2)}(\beta_2)$ where $\pi(x_i) = 
\rho(\beta_i,L)$
defines an order on such a neighbourhood of the identity in $C_S(g)$.
}
an $\Ru$-definable linear order on a neighbourhood of the identity in $C_S(g)$ 
containing $C_{\Soo}(g)$, and the induced linear order on $C_{\Soo}(g)$ makes 
it a linearly ordered abelian group.
Moreover, the order topology on $C_{\Soo}(g)$ coincides with the group 
topology.
We may assume that $g$ is positive with respect to this ordering.

\begin{lemma} \label{l:groupInterval}
  There exists an open interval $J \subseteq  C_{\Soo}(g)$ containing $e$ such that 
  $J$ and the restriction to $J$ of the order on $C_{\Soo}(g)$ are both 
  $\Soo$-definable and $\Ru$-definable.
\end{lemma}
\begin{proof}
  First, consider the $\Ru$-definable set
  $I := g^Sg^S \cap C_S(g)$,
  where $g^Sg^S := \{ g^{a_1}g^{a_2} : a_1,a_2 \in S \}$.
  Since $\Soo$ is normal in $S$, in fact $I \subseteq  C_{\Soo}(g)$.

  Say a subset $X$ of a group is \recalldefn{symmetric} if it is closed 
  under inversion, i.e.\ $X^{-1}=X$.

  \begin{claim}
    $I$ is a closed symmetric interval in $C_{\Soo}(g)$.
  \end{claim}
  \begin{proof}
    Recall that a definable set in an o-minimal structure is 
    \recalldefn{definably connected} if it is not the union of disjoint 
    open definable subsets.

    $X := g^Sg^S$ is the image under a definable continuous map of the 
    definably connected closed bounded set $(g^S)\times(g^S)$, and hence 
    (\cite[1.3.6,6.1.10]{vdD-omin}) $X$ is closed and definably connected.

    Now $X$ is invariant under conjugation by $S$. Thus $\pi(X) \subseteq  \SO_3^{00}$ 
    consists of the rotations around arbitrary axes by the elements of some 
    $\Ru$-definable set $\Theta \subseteq  \SO_2^{00}$, i.e.\ $\pi(X) = 
    \rho(\Theta,S^2(\Ru))$. Since $\rho(\theta,L)^{-1} = \rho(\theta,-L)$, it 
    follows that $\pi(X)$ is symmetric, and hence $\Theta$ is symmetric. 
    Similarly, $\pi(g^S)$ is symmetric, and hence $e \in \pi(X)$ and so $e \in 
    \Theta$.
    
    Recall that if $L$ is the axis of rotation of $\pi(g)$ (i.e.\ $\pi(g) = 
    \rho(\alpha,L)$ for some $\alpha$), then $\pi(C_S(g)) = \rho(\SO_2,L)$. So 
    then $\pi(I)=\rho(\Theta,L)$. Since $\pi(X)$ is closed and definably 
    connected, $\Theta$ is of the form $\Theta' \cup \Theta'^{-1}$ where 
    $\Theta'$ is a closed interval. Thus since $\Theta$ contains the identity, 
    $\Theta$ is itself a closed symmetric interval. So $I$ is a closed 
    symmetric interval in $C_{\Soo}(g)$.
  \end{proof}

  Say $I=[h^{-1},h]$.
  Write the group operation on $C_{\Soo}(g)$ additively.
  For $g_1,g_2 \in \Soo$, let $Y_{g_1,g_2} := g_1^{\Soo} g_2^{\Soo}\cap 
  C_{\Soo}(g)$ (an $\Soo$-definable set). Clearly, $Y_{g_1,g_2}\subset 
  I=[-h,h]$.

  \begin{claim}
    There exist\footnote{
    Although this existence statement suffices for our purposes, in fact one 
    may calculate that we may take $h=g^2$ and $g_1=g=g_2$.
    This can be seen by combining the proof of the claim with the following 
    observations on $\Spin_3$ considered as the group of unit quaternions. 
    Firstly, if $a,b \in \Spin_3(\R)$ have the same scalar part, 
    $\Re(a)=\Re(b)$, then $\Re(a*b) \geq  \Re(a*a)$, with equality iff $a=b$.
    Secondly, conjugation in $\Spin_3$ preserves scalar part.
    Finally, the order on $C_{\Soo}(g)$ agrees (up to inversion) with the 
    order on the scalar part.
    }
    $h'' < h$ and $g_1,g_2 \in \Soo$ such that the interval $(h'',h]$ is 
    contained in $Y_{g_1,g_2}$.
  \end{claim}
  \begin{proof}
    Consider the map $f:S^2 \rightarrow  S$ defined by $f(a_1,a_2) := g^{a_1}g^{a_2}$.
    By definable choice in $\Ru$,
    $f$ admits a $\Ru$-definable section over the set $I$,
    and hence $f$ admits a continuous $\Ru$-definable section on an open 
    interval $\theta : (h',h) \rightarrow  S^2$ for some $h' \in C_{\Soo}(g)$.
    By definable compactness of $S^2$, $\theta$ extends to a continuous 
    $\Ru$-definable section $\theta : (h',h] \rightarrow  S^2$.
    Say $\theta(h) = (a_1,a_2) \in S^2$.
    Let $g_i := g^{a_i}$ for $i=1,2$. Note that $g_i \in \Soo$.

    So by continuity of $\theta$, for some $h'' < h$ we have
    \[ (h'',h] \subseteq  \theta^{-1}((\Soo a_1)\times(\Soo a_2)) \subseteq  f((\Soo 
    a_1)\times(\Soo a_2)) = g_1^{\Soo}g_2^{\Soo} .\]
  \end{proof}

  Thus
  \[ P := h - Y_{g_1,g_2} \subseteq  h - I = [0,2h] \]
  is an $\Soo$-definable subset of the non-negative part of $C_{\Soo}(g)$.

  So set $p := h-h''>0$. Note that $[0,p) \subseteq  P \subseteq  [0,2h]$.
  It is now easy to see that the open interval $(0,p)$ is equal to $P\cap 
  (p-P)$, thus is definable in $\Soo$.

  So $J := (-p,p) = [0,p) \cup -[0,p)$ and its order are $\Soo$-definable, 
  since for $a,b \in J$ we have $a\geq b$ iff $a-b \in [0,p)$.

  Finally, $J = (-p,p)$ and its order are also $\Ru$-definable as an interval 
  in the $\Ru$-definable order on an $\R$-definable neighbourhood of the 
  identity in $C_S(g)$. This ends the proof of Lemma~\ref{l:groupInterval}.
\end{proof}

\subsection{Defining the field}
\label{s:trich}
Let $J$ be as given by Lemma~\ref{l:groupInterval}. We now return to working with 
the full simple compact group $G$, of which $J$ is a subset. Let $n := 
\dim(G)$.
Recall that $J$ consists of
 a neighborhood of the identity in the one-dimensional real algebraic group group $C_S(g)$,
 thus it is a one-dimensional smooth sub-manifold  of $G$. For $h\in G$, the set $J^h$ is an open neighborhood of $e$ in the group $C_S(g)^h$. We let $T_e(J^h)$ denote its tangent space at $e$ with respect to the real closed field $R$.

\begin{lemma} \label{l:chart}
  There exists an open neighbourhood $U \subseteq  \Goo$ of the identity, an open 
  interval $e \in J' \subseteq  J$, and a bijection
  $\phi : J'^n \rightarrow  U$
  which is both $\Goo$-definable and $\Ru$-definable,
  with $\phi(x,e,\ldots ,e) = x$.
\end{lemma}
\begin{proof}
  Let $\g = T_e(G)$ be the Lie algebra of $G$ (as discussed in 
  \secref{nsLie}).
  Consider the $\Ru$-subspace $V \leq  \g(\Ru)$ generated by
  \[ \bigcup_{h \in \Goo} T_e(J^h) = \bigcup_{h \in \Goo} \Ad_h(T_e(J)) .\]
  Then $V$ is $\Ad_{\Goo}$-invariant.
  Thus $\Ad$ restricts to $\Ad : \Goo \rightarrow  \Aut_\Ru(V)$, and hence the 
  differential at the identity is a map $\ad = d_e \Ad : \g(\Ru) \rightarrow  
\End_\Ru(V)$, i.e.\ it follows that $V$ is $\ad_{\g(\Ru)}$-invariant.
  So $V$ is a non-trivial ideal in $\g(\Ru)$.
  But $\g$ is simple, thus we have $V=\g(\Ru)$.

  So say $h_1,h_2,\ldots ,h_n \in \Goo$ are such that $T_e(J^{h_i})$ span $T_eG$.
  Conjugating by $h_1^{-1}$, we may assume $h_1=e$.

  Define $\phi : J^n \rightarrow  \Goo$ by
  \[ \phi(x_1,\ldots ,x_n) := x_1^{h_1}x_2^{h_2}\ldots x_n^{h_n} ,\]
  which is clearly both $\Goo$-definable and $\Ru$-definable.
  Then the differential $d_{(e,\ldots ,e)}\phi : T_eJ^n \rightarrow  T_eG(\Ru)$ is an 
  isomorphism.
  Thus by the implicit function theorem (for the real closed field $\Ru$),
  for some open interval $e \in J' \subseteq  J$,
  the restriction $\phi\negmedspace\restriction_{(J')^n}$ is a bijection with some open neighbourhood 
  $U$ of $e$,
  as required.

\end{proof}

Redefine $J$ to be $J'$ as provided by the lemma.

We shall consider the o-minimal structure obtained by expanding the interval 
$(J;<)$ by the pullback of the group operation near $e$ via the chart $\phi$.
As we will now verify, it follows from the non-abelianity of $G$ that the 
resulting structure on $J$ is ``rich'' in the sense of the o-minimal 
trichotomy theorem \cite{PS-trichotomy} (see below), and so by that theorem it 
defines a field. Let us first recall the relevant notions:

\begin{definition} Let $\mathcal M=( M;<,\cdots)$ be an o-minimal structure. A \emph{definable family of curves}
is given by a definable set $F\sub M^n\times T$, for some definable $T\sub M^k$, such that for every $t\in T$,
the set $F_t=\{a\in M^n:(a,t)\in F\}$ is of dimension $1$.

The family is called {\em normal} if for $t_1\neq t_2$, the set $F_{t_1}\cap F_{t_2}$ is finite. In this case,
the dimension of the family is taken to be $\dim T$.

A linearly ordered set $(I;<)$, together with a partial binary function $+$ and a constant $0$,  is called {\em a group-interval} if $+$ is continuous, definable
in a neighborhood of $(0,0)$,  associative and commutative when defined, order preserving in each coordinate, has $0$ as a neutral element and each element near $0$ has additive inverse in $I$.
\end{definition}

We shall use   Theorem 1.2 in \cite{PS-trichotomy}:
\begin{fact} \label{trich} Let $\mathcal I=(I;<,+,0,\cdots)$ be an 
  $\omega_1$-saturated o-minimal expansion of a group-interval.
Then one and only one of the following holds:
\begin{enumerate}
\item There exists an ordered vector space $\mathcal V$ over an ordered 
  division ring, such that $(I;<,+,0)$ is definably isomorphic to a 
  group-interval in $V$ and the isomorphism takes every $\mathcal I$-definable 
  set to a definable set in $\mathcal V$.

\item A real closed field is definable in $\mathcal I$, with its underlying set a subinterval of $I$ and its ordering compatible with $<$.
\end{enumerate}
\end{fact}

We now return to our interval $J\sub C_S(g)\sub G^{00}$ and to the definable bijection $\phi:J^n\to U$.
We  let $\star:J^n\times J^n\dashrightarrow J^n $ be the partial function obtained as the pullback via $\phi$
of the group operation in $G$. Namely, for $a,b,c\in J^n$, $$a\star b=c\Leftrightarrow \phi(a)\cdot \phi(b)=\phi(c).$$
Since $\phi$ and the group operation on $U$ are definable in both $\Goo$ and $\Ru$, then so is $\star$.
Because $J$ is an open interval around the identity inside the one-dimensional group $C_S(g)\sub G$,
the restriction to $J$ of the group operation of $G$ makes $J$ into a group-interval, and we let $+$ denote this restriction to $J$. Note that the ordering on $J$ is definable using $+$ and therefore definable in both $\Goo$ and in $\mathcal R$.





\begin{lemma} \label{l:nonlinear}
   The structure $\mathcal J=(J;<,+,\star)$ is an o-minimal expansion of a 
   group-interval, not of type (1) in the sense of Fact \ref{trich}. It is definable in both $\Goo$ and in  $\mathcal R$.
\end{lemma}
\begin{proof}
  The structure $\mathcal J$ is o-minimal since it is definable in the 
  o-minimal structure $\Ru$ and the ordered interval $(J;<)$ is definably 
  isomorphic via a projection map with an ordered interval in $(\Ru;<)$.
  We want to show that $\mathcal J$ is not of type (1).

  To simplify notation we denote below the group $C_S(g)$ by $H$. Because $G$ is a simple group, there are infinitely many  distinct conjugates of the one-dimensional group $H$. More precisely, $\dim N_G(H)<\dim G$ and if $h_1,h_2\in G$ are not in the same right-coset of $N_G(H)$ then $H^{h_1}\cap H^{h_2}$ and hence also $J^{h_1}\cap I^{h_2}$ is finite.  Since $\dim (N_G(H))<\dim G$ one can find infinitely many $h\in G$, arbitrarily close to $e$, which belong to different right-cosets of $N_G(H)$. In fact, by Definable Choice in o-minimal expansions of groups or group-intervals, we can find in $\mathcal J$ a definably connected one-dimensional set $C\sub J^n$ with $(e,\ldots,e)\in C$,  such that no two elements of $\phi(C)$ belong to the same right-coset of $N_G(H)$.

  The  family $\{H^h\cap U:h\in \phi(C)\}$ is a one-dimensional normal family of definably connected curves in $U$, all containing $e$, and we want to ``pull it back'' to the structure $\mathcal J$. In order to do that we first note that by replacing $J$ by a subinterval $J_0\sub J$, and replacing $C$ by a possibly smaller definably connected set, we may assume that for every $h\in \phi(C)$ and every $g\in J_0$, each of $h, h^{-1}g$ and $ h^{-1}gh$  is inside $U$. Thus, the one-dimensional normal family of definably connected curves
$$\{\phi^{-1}(J_0^h):h\in \phi(C)\}$$
is definable in $\mathcal J$, and all of these curves contain the point $(e,\ldots,e)$
(=$\phi^{-1}(e))$.

 Now, if $\mathcal J$ was of type (1) in Fact \ref{trich}, then up to a change of signature it would be a reduct of an ordered vector space. However, it easily follows from quantifier elimination in ordered vector spaces that in such structures there is no definable infinite normal family of one-dimensional definably connected sets, all going through the same point.
  Hence, $\mathcal J$ is not of type (1).
\end{proof}

By Fact ~\ref{trich},
there is a $\mathcal J$-definable real closed field $K$ on an open interval
in $J$ containing $e$. Thus $K$ with its field structure is also 
$\Goo$-definable and $\Ru$-definable.


\subsection{Obtaining the valuation, and bi-interpetability}
\label{s:Goo}
\newcommand{\GooK}{G_K^{00}}
Let $K \subseteq  J$ be the definable real closed field obtained in the previous 
subsection, considered as a pure field.
By \cite{OPP-groupsRings},
there is an $\Ru$-definable field isomorphism $\theta_R : \Ru \rightarrow  K$.
Let $G_K := G^{\theta_R}(K)$ be the group of $K$-points of the $K$-definable 
(linear algebraic) group obtained by applying $\theta_R$ to the parameters 
defining $G$, so
$\theta_R$ induces an $\Ru$-definable group isomorphism $\theta_G : G \rightarrow  
G_K$.
Let $\GooK := (G^{\theta_R})^{00}(K)$ be the corresponding infinitesimal 
subgroup,
thus $\theta_R$ restricts to an isomorphism $\theta_G\negmedspace\restriction_\Goo : \Goo \rightarrow  \GooK$.

Denote by $(\Ru;\Goo)$
the expansion of the field $\Ru$ by a predicate for $\Goo \leq  G$.

\begin{lemma} \label{l:rcvfBiterp}
  $(\Ru;\Goo)$ and $(\Ru;\O)$ have the same definable sets.
\end{lemma}
\begin{proof}
  Since $G$ is defined over $\R$,
  it admits a chart at the identity defined over $\R$,
  that is, an $\R$-definable homeomorphism $\psi : I^n \rightarrow  U$,
  where $I \subseteq  \Ru$ is an open interval around $0$,
  and $U \subseteq  G$ is an open neighbourhood of $e$,
  and $\psi((0,\ldots ,0))=e$.

  Then $\st_G(\psi(x)) = \psi(\st(x))$,
  and so $\Goo = \psi(\m^n)$.
  Thus $\Goo$ is definable in $(\Ru;\O)$, and conversely $\m^n$, and hence 
  $\m$, and so also $\O = \Ru \setminus \frac1\m$, are definable in $(\Ru;\Goo)$.
\end{proof}

\begin{lemma} \label{l:gookdef}
  $\GooK \leq  G_K$ is $\Goo$-definable,
  and moreover $\theta_G\negmedspace\restriction_\Goo : \Goo \rightarrow  \GooK$ is $\Goo$-definable.
\end{lemma}
\begin{proof}

  Let $n := \dim(G)$.

  Precisely as in \cite[3.2.2]{PPS-defblySimple},
  translating by an element of $\Goo$ if necessary we may assume that the 
  group operation is $C_1$ for $K$ on a neighbourhood of the identity 
  according to the chart $\phi\negmedspace\restriction_{K^n}$ (where $\phi$ is the map from 
  Lemma~\ref{l:chart}),
  and then the adjoint representation yields an $\Ru$-definable homomorphism
  $\Ad : G \rightarrow  \GL_n(K)$.
  Since $\Goo$ defines $\phi$ and the field $K$ and the conjugation maps $x 
  \mapsto  x^g$ for $g,x \in \Goo$, the restriction
  $\Ad\negmedspace\restriction_{\Goo} : \Goo \ensuremath{\lhook\joinrel\relbar\joinrel\rightarrow} \GL_n(K)$
  is $\Goo$-definable, and is an embedding since $G$ has finite centre, and $\Goo$ is torsion-free.

  Define $\eta := \Ad \o \theta_G^{-1} : G_K \ensuremath{\lhook\joinrel\relbar\joinrel\rightarrow} \GL_n(K)$.
  So $\eta$ is $\Ru$-definable.
  Since $K$ is $\Ru$-definably isomorphic to $\Ru$, the $\Ru$-definable 
  structure on $K$ is just the field structure.
  Thus $\eta$ is also $K$-definable, and hence $\Goo$-definable.

  So $\theta_G\negmedspace\restriction_{\Goo} = \eta^{-1} \o \Ad\negmedspace\restriction_{\Goo}$ is $\Goo$-definable.

\end{proof}

\begin{proof}[Proof of Theorem~\ref{t:mainBiterp}]
  By Lemma~\ref{l:gookdef}, $\theta_R$ provides a definition of $(\Ru;\Goo)$ in 
  $\Goo$ with universe $K$.
  This forms a bi-interpretation together with the tautological interpretation 
  of $\Goo$ in $(\Ru;\Goo)$;
  indeed, the composed interpretations are $\theta_R$ and $\theta_G\negmedspace\restriction_{\Goo}$,
  which are $\Ru$-definable and $\Goo$-definable respectively.

  Combining this with Lemma~\ref{l:rcvfBiterp} concludes the proof of 
  Theorem~\ref{t:mainBiterp}.
\end{proof}

\section{Isomorphisms of infinitesimal subgroups}
\label{s:borel-tits}

Cartan \cite{cartan} and van der Waerden \cite{waerden} showed that any 
abstract group isomorphism between compact semisimple Lie groups is 
continuous. In a similar spirit, Theorem~\ref{t:borelTits} below shows that every 
abstract group isomorphism of two infinitesimal subgroups of simple compact Lie 
groups is, up to field isomorphisms, given by an algebraic map.

We preface the proof with two self-contained preliminary subsections. In 
outline, the proof is as follows. We first prove in Theorem~\ref{t:defbleFieldsRCVF} 
that given $R \vDash  \RCVF$, any model of $\RCVF$ definable $R$ is definably 
isomorphic to $R$. As in other cases of the ``model-theoretic Borel-Tits 
phenomenon'', first described for $\ACF$ in \cite{poizat-BT}, it follows that 
every abstract group isomorphism of the infinitesimal subgroups is the 
composition of a valued field isomorphism with an $\RCVF$-definable group 
isomorphism; we give a general form of this argument in 
Lemma~\ref{l:abstractNonsenseBT}. Finally, in Section~\ref{s:borelTits} we deduce the 
final statement by seeing that any $\RCVF$-definable group isomorphism of the 
infinitesimal subgroups is induced by an algebraic isomorphism of the Lie 
groups.

\subsection{Definable fields in $\RCVF$}
Here, we show that there are no unexpected definable fields in $R^n$ for $R \vDash  
\RCVF$.

In this subsection we reserve the term `semialgebraic' for $R$-semialgebraic 
sets.
Let $\Ru_v=\langle R;\O\rangle \vDash  \RCVF$.
We say a point $a$ of an $\Ru_v$-definable set $X$ over $A$ is {\em generic} 
over $A$ if $\trd(a/A)$ is maximal for points in $X(R')$ for $R'$ an 
elementary extension of $R$. Such an $a$ exists if $R$ is 
$(\aleph_0+|A|^+)$-saturated. This maximal transcendence degree is the 
dimension $\dim(X)$ of $X$, which coincides 
(\cite[Theorem~4.12]{MMS-weakOMin}) with the largest $d$ such that the image 
of $X$ under a projection to $d$ co-ordinates has non-empty interior.

\begin{fact} \label{f:semialgebraic}\begin{enumerate}\item Let $X\subseteq 
    R^n$ be an $\Ru_v$-definable set over $A$, and $a\in X$ a generic element 
    over $A$.
Then there exists a semialgebraic neighborhood $U\subseteq R^n$ of $a$ (possibly defined over additional parameters, which may be taken to be independent of $Aa$) such that $U\cap X$ is semialgebraic.

\item Let $U\subseteq R^d$ be an open $\Ru_v$-definable set over $A$, and $a\in U$ generic over $A$. Let $F:U\to R$ be an $\Ru_v$-definable function. Then there exists a semialgebraic neighborhood $U$ of $a$, such that $F| U$ is semialgebraic and $C^1$ with respect to $R$, meaning that all partial derivatives of $F$ with respect to $R$ exist and are continuous on $U$.

\end{enumerate}
\end{fact}
\proof
\begin{enumerate}
  \item
    Permuting co-ordinates, we may assume $a=(b,c)$ where $b \in R^d$ is 
    generic over $A$ and $c \in R^{n-d}$ is in $\dcl(bA)$.
    Let $\pi : R^n \rightarrow  R^d$ be the projection to the first $d$ co-ordinates.

    By \cite[Theorem~4.11]{MMS-weakOMin}, $X$ admits a decomposition 
    into finitely many disjoint $A$-definable cells each of which is the graph of a 
    definable function on an open subset of some $R^t$.

    Definable closure in $\Ru_v$ coincides with definable closure in $R$ 
    \cite[Theorem~8.1(1)]{Mellor-RCVFEI}.
    Thus if $C$ is the cell containing $a$, then $C$ is the graph of a 
    semialgebraic function on a neighbourhood of $b$. We now claim that 
    locally near $a$, the set $X$ is equal to $C$.

    Suppose for a contradiction that $C' \neq C$ is another cell in the
    decomposition, and $a \in \cl(C')$, the topological closure of $C'$.
    The cell decomposition of $X$ induces a cell decomposition of $\pi(X)\sub R^d$, and since $b$ is generic in $R^d$
    over $A$, it must belong to the interior of $\pi(C)$. It follows that $\pi(C')=\pi(C)$, and so there exists $c'\neq c$ such that $(b,c')\in C'$.

       By the inductive definition of a cell
    and the genericity of $a$ in $X$, also $C'$ is the graph of a
    semialgebraic function on a neighbourhood of $b$, thus in particular is 
    locally closed, contradicting $a \in \cl(C')$.

  \item By (1), the graph of $F$ is a semialgebraic set in a neighborhood of 
    $(a,f(a))$, and since $a$ is generic in its domain, the function $F$ is 
    $C^1$ in a neighborhood of $a$.
\end{enumerate}
\qed

\begin{theorem}\label{t:defbleFieldsRCVF} Let $\Ru_v=\langle R;\O\rangle \vDash  
  \RCVF$.
  \begin{enumerate}[(a)]\item
    If $F \subseteq  R^n$ is a definable field in $\Ru_v$ then it is $\Ru_v$-definably 
    isomorphic to either $R$ or its algebraic closure $R(\sqrt{-1})$.
  \item
    If $F \subseteq  R^n$ is a definable non-trivially valued field in $\Ru_v$ then it 
    is $\Ru_v$-definably isomorphic to either $\Ru_v$ or its algebraic closure 
    $\Ru_v(\sqrt{-1})$.
  \end{enumerate}
\end{theorem}
\proof
\newcommand{\la}{\langle}
\newcommand{\ra}{\rangle}
Passing to an elementary extension as necessary, we assume $R$ is 
$(\aleph_0+|A|^+)$-saturated for any parameter set $A$ we consider.

We first prove (a). Let $d=\dim F$.
 We first show that the additive group of $F$ can be endowed with the structure of a definable $C^1$ atlas (not necessarily finite). By that we mean: a definable family of subsets of $F$, $\{U_t:t\in T\}$, and a definable family of bijections $f_t:U_t\to V_t$, where each $V_t$ is an open subset of $R^{d}$, such that for every $s, t\in T$, the set $f_t(U_t\cap U_s)$ is open in $V_t$ and the transition maps $\sigma_{t,s}: f_t(U_t\cap U_s)\to f_s(U_t\cap U_s)$
are $C^1$ with respect to $R$. Moreover, the group operation and additive inverse are $C^1$-maps when read through the charts.

To see this we follow the strategy of the paper of Marikova, \cite{Marikova}. Without loss of generality $F$ is definable over $\emptyset$. We fix $g\in F$ generic and an open neighborhood $U\ni g$ as in Fact \ref{f:semialgebraic} (1).  By the cell decomposition in real closed fields, we may assume that $U\cap F$ is a cell, so definably homeomorphic to some open subset $V$ of $R^{d}$. By replacing $U\cap F$ with $V$ we may assume that $U$ is an open subset of $R^d$, and $g$ is generic in $F$ over the parameters defining $U$.

\begin{claim} \label{fact2.10} The map $(x,y,z)\mapsto x-y+z$ is a $C^1$-map (as a map from $U^3$ into $U$) in some neighborhood of $(g,g,g)$.\end{claim}
\proof The proof is identical to \cite[Lemma 2.10]{Marikova}, with Fact \ref{f:semialgebraic} (2) above replacing Lemma 2.8 there.\qed

Thus, there exists $U_0\ni g$, such that the map $(x,y,z)\mapsto x-y+z$ is a $C^1$ map from $U_0^3$ into $U$. 
We now consider the definable cover of $F$:
$$\U=\{h+U_0:h\in F\},$$ (with $+$ the $F$-addition) and the associated family of chart maps $f_h:h+U_0\to U_0$, $f_h(x)=x-h$.
Using Claim \ref{fact2.10}, it is not hard to see that $\U$ endows $\la F,+\ra$ with a definable $C^1$-atlas;
indeed, if $h+U_0 \cap h'+U_0 \neq \emptyset$, say $h+u_0 = h'+u_0'$,
then $h-h' = u_0' - u_0$,
so $\sigma_{h,h'}(u_0'') = h + u_0'' - h' = u_0'' - u_0 + u_0'$.
Similarly, the function $+$ is a $C^1$-map from $F^2$ into $F$ (where $F^2$ is endowed with the product atlas), and $x\mapsto -x$ is a $C^1$-map as well. Indeed, in \cite{Marikova} Marikova proves in exactly the same way that the same $\U$ endows the group with a topological group structure (using \cite[Lemma~2.10]{Marikova} in place of Claim \ref{fact2.10}).

By Fact \ref{f:semialgebraic} (2), every definable function from $F$ to $F$ is $C^1$ in a neighborhood of generic point of $F$. Thus, just as in \cite[Lemma 2.13]{Marikova}, we have:
\begin{fact}\label{endo} If $\alpha :F\to F$ is a definable endomorphism of $\la F,+\ra$ then $\alpha$ is a $C^1$-map.\end{fact}

For every $c\in F$, we consider the map $\lambda_c:F\to F$, defined by $\lambda_c(x)=c x$ (multiplication in $F$). 
By fact \ref{endo} each $\lambda_c$ is a $C^1$-map and we consider its Jacobian matrix at $0$, with respect to $R$, denoted by $J_0(\lambda_c)$. This is a matrix in $M_d(R)$, and the map $c\mapsto J_0(\lambda_c)$ is $\Ru_v$-definable.

As was discussed in \cite[Lemma 4.3]{OPP-groupsRings}, it follows from the chain rule that the map $c\mapsto J_0(\lambda_c)$ 
is a ring homomorphism into $M_d(R)$ (note that we do not use here the uniqueness of solutions of ODE as in \cite{OPP-groupsRings}, thus we a-priori only obtain a ring homomorphism). However, since $F$ is a field the map is injective. 

To summarize, we mapped $F$ isomorphically and definably onto an $\Ru_v$-definable field, call it $F_1$, of matrices inside $M_d(R)$. Notice that now the field operations are just the usual matrix operations, $1_{F_1}$ is the identity matrix, so in particular, all non-zero elements of $F_1$ are invertible matrices in $M_d(R)$. Our next goal is to show that $F_1$ is semialgebraic. 

By Fact \ref{f:semialgebraic} (2), there exists some non-empty relatively open subset of $F_1$ which is semialgebraic.
By translating it to $0$ (using now the semialgebraic $F_1$-addition), we find such a neighborhood, call it $W\subseteq F_1$, of the $0$-matrix. But now, given any $a\in F$, by multiplying $a$ by an invertible matrix $b\in W$ sufficiently close to $0$, we obtain $ba\in W$. Thus, $F_1=\{a^{-1}b: a,b\in W\}$. Because $W$ is semialgebraic so is $F_1$.

Thus, we showed that $F$ is definably isomorphic in $\Ru_v$ to a semialgebraic field $F_1$. We now apply Theorem \cite[Theorem 1.1]{OPP-groupsRings} and conclude that $F_1$ is semialgebraically isomorphic to $R$ or to $R(\sqrt{-1})$.

  Finally, we address (b).
  This follows immediately from (a) once we observe that $\O$ is the only 
  definable valuation ring in $\Ru_v$. So suppose $\O'$ is another. By weak 
  o-minimality, $\O'$ is a finite union of convex sets, and then since it is a 
  subring with unity it is easy to see that $\O'$ is convex. Thus either $\O 
  \subseteq  \O'$ or $\O' \subseteq  \O$. Without loss of generality, $\O$ is the standard 
  valuation ring $\cup_{n \in \N} [-n,n]$, so $\O$ properly contains no 
  non-trivial convex valuation ring. Thus $\O \subseteq  \O'$. But then if $v : R \rightarrow  
  \Gamma$ is the valuation induced by $\O$, then the image of the units of 
  $\O'$ is a definable subgroup $v((\O')^*) \leq  \Gamma$. But $\Gamma$ is a pure 
  divisible ordered abelian group, and so has no non-trivial definable 
  subgroup. Hence $v((\O')^*) = \{0\}$, and hence $\O' = \O$.
\qed

\begin{remark}
  The techniques we applied here will not readily adapt to handle imaginaries.
  In the case of algebraically closed valued fields, a result of 
  \cite{HR-metastable} is that the only {\em interpretable} fields, up to 
  definable isomorphism, are the valued field and its residue field. It would 
  be natural to expect that, correspondingly, the only interpretable fields in 
  $\RCVF$ up to definable isomorphism are the valued field, its residue field, 
  and their algebraic closures.
\end{remark}

\subsection{Interpretations and general nonsense}
We address the ``Borel-Tits phenomenon'' associated with bi-interpretations 
which require parameters. We spell out an abstract form of the argument given 
by Poizat \cite{poizat-BT} in the case of algebraically closed fields. The 
ideas in this subsection are well-known. For convenience of exposition, we 
first give a name to the following key property.

\begin{definition} \label{d:self-recollecting}
  Say a theory $T$ is \defn{self-recollecting} if any $\B' \models T$ 
  interpreted in any $\B \models T$ is $\B$-definably isomorphic to $\B$.

  Say $T$ is \defn{self-recollecting for definitions} if this holds for
  interpretations which are definitions (where recall a {\em definition} is an 
  interpretation which doesn't involve non-trivial quotients).
\end{definition}

\begin{examples}
  $\ACF$ is self-recollecting by \cite{poizat-BT}, $\RCF$ is self-recollecting 
  by \cite{NP-compactLie}, and $\Th(\Q_p)$ is self-recollecting for 
  definitions by \cite{pillay-fieldsQp}. It follows directly from the 
  characterisation of interpretable fields in \cite{HR-metastable} that the 
  theory of non-trivially valued algebraically closed fields $\ACVF$ is 
  self-recollecting.

  Theorem~\ref{t:defbleFieldsRCVF}(b) proves that $\RCVF$ is self-recollecting for 
  definitions, but we do not settle the question of whether it is 
  self-recollecting.
\end{examples}

We use the notation $\alpha : \terp{\A}{\B}$ to denote an interpretation of 
$\A$ in $\B$, which recall we consider to be a map from $\A$ to some definable 
quotient in $\B$. Note that any isomorphism is in particular an interpretation 
(and even a definition). We denote composition of interpretations by 
concatenation.

\begin{lemma} \label{l:abstractNonsenseBT}
  Suppose $\A_i$ is a structure interpreted in a structure $\B_i$ for $i=1,2$,
  and the interpretation of $\A_1$ in $\B_1$ can be completed to a 
  bi-interpretation.
  Suppose further that $\Th(\B_1) = \Th(\B_2)$, and $T_B := \Th(\B_1)$ is 
  self-recollecting.

  Suppose $\sigma : \A_1 \rightarrow  \A_2$ is an isomorphism of structures.

  Then there exist an isomorphism $\sigma' : \B_1 \xrightarrow{\cong} \B_2$
  and a $\B_2$-definable isomorphism $\theta : \sigma'(\A_1) \rightarrow  \A_2$
  such that $\sigma = \theta(\sigma'\negmedspace\restriction_{\A_1})$.

  If $T_B$ is only self-recollecting for definitions but the given 
  interpretations are definitions, then the same result holds.
\end{lemma}
\begin{proof}
  Let $(f,g)$ be the bi-interpretation of $\A_1$ with $\B_1$,
  and $\alpha : \terp{\A_2}{\B_2}$ the interpretation.

  \[ \xymatrix{
  \A_1 \biterpAr{d}{f} \ar[r]^\sigma & \A_2 \terpAr[d]^{\alpha} \\
  \B_1 \biterpAr{u}{g} & \B_2 \\
  } \]

  For interpretations $\beta,\gamma : \terp{\A}{\B}$, we write $\beta \sim 
  \gamma$ if
  $\gamma \beta^{-1} : \beta(\A) \rightarrow  \gamma(\A)$ is a $\B$-definable 
  isomorphism between the two copies of $\A$ in $\B$.\footnote{Interpretations 
  satisfying this condition are sometimes called {\em homotopic}.}

  Now $\alpha \sigma g$ is an interpretation of $\B_1$ in $\B_2$,
  and thus by self-recollecting
  there is a $\B_2$-definable isomorphism $\tau : (\alpha\sigma g)(\B_1) \rightarrow  
  \B_2$. Let $\sigma' := \tau\alpha\sigma g : \B_1 \rightarrow  \B_2$.
  Then $\alpha \sigma g \sim \sigma'$, and so $\alpha \sigma gf \sim \sigma' 
  f$.

  Now $gf$ is definable,
  and it follows that $\alpha \sigma \sim \alpha \sigma gf$.
  Thus $\alpha \sigma \sim \sigma' f$.
  Then $\alpha' := \sigma' f \sigma^{-1} \sim \alpha$.


  So $\theta := \alpha\alpha'^{-1}$ is a $\B_2$-definable isomorphism,
  and $\theta \sigma' f = \alpha \sigma$. Since we view $\A_1$ and $\A_2$ in 
  $\B_2$ via $f$ and $\alpha$, this is as desired.

  The proof in the case of definitions is identical.
\end{proof}

\subsection{Characterisation of isomorphisms of infinitesimal subgroups}
\label{s:borelTits}

\begin{lemma} \label{l:isomExtend}
  If $G,H$ are compact connected centreless linear Lie groups,
  and $\Ru \succ  \R$ is a proper real closed field extension of $\R$,
  and $\theta : \Goo(\Ru) \xrightarrow{\cong} \Hoo(\Ru)$ is an $(\Ru;\O)$-definable group 
  isomorphism,
  then $\theta$ extends to an $\Ru$-definable algebraic isomorphism $G(\Ru) 
  \xrightarrow{\cong} H(\Ru)$.
\end{lemma}
\begin{proof}
  We may assume $\Ru$ is $\aleph_0$-saturated.
  Let $\Gamma \subseteq  (G\times H)(\Ru)$ be the $\Ru$-Zariski closure of the graph 
  $\Gamma_\theta$ of $\theta$.
  Since $\Gamma_\theta$ is an abstract subgroup, $\Gamma$ is (the set of 
  $\Ru$-points of) an algebraic subgroup over $\Ru$.
  The image of the projection $\pi_G : \Gamma \rightarrow  G(\Ru)$ contains 
  $\Goo(\Ru)$, which is Zariski-dense in $G$ since $G$ is connected, and thus 
  $\pi_G(\Gamma) = G(\Ru)$.

  Let $k \leq  \Ru$ be a finitely generated field over which $\theta$ is defined. 
  Since $\Gamma_\theta$ is Zariski dense in $\Gamma$, there exists 
  $(g,\theta(g)) \in \Gamma_\theta$ which is algebraically generic in $\Gamma$ 
  over $k$ (i.e.\ of maximal transcendence degree);
  indeed, the Zariski density implies that $\Gamma_\theta$ is contained in no 
  subvariety of $\Gamma$ over $k$ of lesser dimension, and so such a generic 
  exists by $\aleph_0$-saturation of $\Ru$.
  But definable closure in $\RCVF$ agrees with field-theoretic algebraic 
  closure \cite[Theorem~8.1(1)]{Mellor-RCVFEI}, so $\trd(\theta(g)/k(g))=0$ 
  and thus $\pi_G^{-1}(g)$ is finite.
  Then also $\ker(\pi_G)$ is finite, and hence central.
  So since $G$ is centreless, $\pi_G$ is an isomorphism.

  Similarly, $\pi_H$ is an isomorphism.
  So $\theta$ extends to the algebraic isomorphism $\pi_H \o \pi_G^{-1}$.
\end{proof}

\begin{theorem} \label{t:borelTits}
  Suppose $G_1$ and $G_2$ are compact simple centreless linear Lie groups,
  and $\Ru_i \succ  \R$ is a proper real closed field extension of $\R$ for 
  $i=1,2$,
  and $\sigma : G_1^{00}(\Ru_1) \xrightarrow{\cong} G_2^{00}(\Ru_2)$ is an abstract group 
  isomorphism.

  Then there exist
  a valued field isomorphism $\sigma' : (\Ru_1,\O_1) \xrightarrow{\cong} (\Ru_2,\O_2)$
  and an $\Ru_2$-definable isomorphism $\theta : G_3(\Ru_2) \xrightarrow{\cong} G_2(\Ru_2)$,
  where $G_3 = \sigma'(G_1)$,
  such that $\sigma = \theta\negmedspace\restriction_{G_3^{00}} \o \sigma'\negmedspace\restriction_{G_1^{00}}$.

  In particular, $\sigma$ extends to an abstract group isomorphism $G_1(\Ru_1) 
  \xrightarrow{\cong} G_2(\Ru_2)$.
\end{theorem}
\begin{proof}
  By Theorem~\ref{t:mainBiterp}, Theorem~\ref{t:defbleFieldsRCVF}(b), and 
  Lemma~\ref{l:abstractNonsenseBT}, there exist an isomorphism $\sigma' : 
  (\Ru_1,\O_1) \xrightarrow{\cong} (\Ru_2,\O_2)$ and an $(\Ru_2,\O_2)$-definable isomorphism 
  $\theta' : \sigma'(G_1^{00}(\Ru_1)) \xrightarrow{\cong} G_2^{00}(\Ru_2)$
  such that $\sigma = \theta' \o \sigma'\negmedspace\restriction_{G_1^{00}}$.

  Then $\sigma'(G_1^{00}(\Ru_1)) = G_3^{00}(\Ru_2)$.
  Thus by Lemma~\ref{l:isomExtend}, $\theta'$ extends to an $\Ru_2$-definable 
  algebraic isomorphism $\theta : G_3(\Ru_2) \xrightarrow{\cong} G_2(\Ru_2)$, as required.
\end{proof}

\begin{remark}
  We have stated the results of this section in terms of $\Goo$, but it is 
  easy to see that they apply equally to other infinitesimal subgroups as in 
  Remark~\ref{r:otherValns}.
\end{remark}

\section{Infinitesimal subgroups of definably compact groups}
\label{s:DC}

In this section, we prove Theorem~\ref{t:DCBiterp} by combining Theorem~\ref{t:mainBiterp} 
with results in the literature on definably compact groups and $\Goo$.

We work in a sufficiently saturated o-minimal expansion $M$ of a real closed 
field, say $\kappa$-saturated where $\kappa$ is sufficiently large. (In fact 
$\kappa=2^{\aleph_0}$ suffices for the arguments below; moreover, it follows 
after the fact that Theorem~\ref{t:DCBiterp} holds with only $\kappa=\aleph_0$, but 
we do not spell this out).
For $G$ a definable group, let $\Goo$ be the smallest $\bigwedge$-definable 
(in the sense of $M$) subgroup of bounded index. Here, a 
\recalldefn{$\bigwedge$-definable set} is a set defined by an infinite 
conjunction of formulas over a common parameter set $A \subseteq  M$ with $|A| < 
\kappa$.

\begin{lemma} \label{l:oostruct}
  Let $G$ and $H$ be definably compact definable groups.
  \begin{enumerate}[(i)]\item If $\theta : G \twoheadrightarrow  H$ is a definable surjective homomorphism, then 
  $\theta(\Goo) = \Hoo$.
  \item If $H$ is a definable subgroup of $G$, then $\Hoo = \Goo \cap H$.
  \item $\Goo$ is the unique $\bigwedge$-definable subgroup of bounded index 
  which is divisible and torsion-free.
  \item $(G\times H)^{00} = \Goo \times \Hoo$
  \end{enumerate}
\end{lemma}
\begin{proof}
  \begin{enumerate}[(i)]\item $\theta^{-1}(\Hoo)$ resp.\ $\theta(\Goo)$ is a $\bigwedge$-definable 
  bounded index subgroup of $G$ resp.\ $H$.
  \item \cite[Theorem~4.4]{B-spectra}.
  \item \cite[Corollary~4.7]{B-spectra}.
  \item This follows directly from (iii).
  \end{enumerate}
\end{proof}

Say a group $(G;*)$ is the \defn{definable internal direct product} of 
its subgroups $H_1,\ldots ,H_n$ if each $H_i$ is $(G;*)$-definable and 
$(h_1,h_2,\ldots ,h_n) \mapsto  h_1h_2\ldots h_n$ is an isomorphism $\prod_i H_i \rightarrow  G$.
The following lemma is an immediate consequence of the definition.

\begin{lemma} \label{l:biterpProd}
  If a group $(G;*)$ is the definable internal direct product of subgroups 
  $H_1,\ldots ,H_n$,
  then $(G;*)$ is bi-interpretable with the disjoint union of the $(H_i;*)$.
\end{lemma}

The following Fact extracts from the literature the key results we will need 
on the structure of definably connected definably compact groups.
A definable group is \recalldefn{definably simple} if it contains no proper 
non-trivial normal definable subgroup.
First recall that if $G\sub M^n$ is a definable group in an o-minimal 
structure $M$ then, by \cite{P-groupsFieldsOMin}, it admits a topology with a 
definable basis which makes it into a topological group. Moreover, this is the 
unique topology which agrees with the ambient $M^n$-topology on a definable 
subset of $G$ whose complement has smaller dimension. All topological notions 
below (e.g.\ definable compactness) are with respect to this topology.

\begin{fact} \label{f:DCLit}
  Let $G$ be a definably connected definably compact definable group.
  Let $G' := (G,G)$ be the derived subgroup of $G$. Then:

  \begin{enumerate}[(i)]\item $G'$ and $Z(G)$ are definable and definably compact.
  \item $G$ is the product of its subgroups $G'$ and $Z(G)^0$,
    and $Z(G)^0 \cap G'$ is finite.
  \item $Z(G')$ is finite, and $G'/Z(G')$ is the direct product of finitely 
  many definably simple definably compact definable subgroups $H_i$.
  \item If $G$ is definably simple, then there exists a compact real linear 
  Lie group $H$ and a real closed field $\Ru$ extending $\R$ such that 
  $(\Goo;*)$ is isomorphic to $(\Hoo;*)$, where $\Hoo := \Hoo(\Ru)$ is the 
  infinitesimal neighbourhood of the identity as defined in \secref{intro}.
  \end{enumerate}
\end{fact}
\begin{proof}
  \begin{enumerate}[(i)]\item By \cite[Corollary~6.4(i)]{HPP-central}, $G'$ is definable. Since 
  definable subgroups are closed, both groups are definably compact.
  \item This is \cite[Corollary~6.4(ii)]{HPP-central}.
  \item This is immediate from \cite[Corollary~6.4(i)]{HPP-central} and 
  \cite[Fact~1.2(3)]{HPP-central} (based on 
  \cite[Theorem~4.1]{PPS-defblySimple}).
  \item This follows from the proof of 
  \cite[Proposition~3.6]{Pillay-conjecture}. Indeed, as discussed there (see 
  also \cite[Fact~1.2(1)]{HPP-central}), $G$ is definably isomorphic to $H(R)$ 
  for $R$ a definable real closed field and $H$ a semialgebraic linear group 
  over a copy of the real field within $R$. Since $G$ is definably compact, 
  $H$ is compact (see \cite[Theorem~2.1]{PS-defbleCompactness}). Now {\em Case 
  II} in the proof of \cite[Proposition~3.6]{Pillay-conjecture} shows that the 
  smallest $M$-$\bigwedge$-definable subgroup $H^{00}(R)$ of $H$ is precisely the 
  infinitesimal neighbourhood $\st^{-1}(e)$, as required.
  \end{enumerate}
\end{proof}

We now repeat the statement of Theorem~\ref{t:DCBiterp}, and prove it.
{
  \renewcommand{\thetheorem}{\ref{t:DCBiterp}}
  \begin{theorem}
    Let $(G;*)$ be an infinite definably compact group definable in a 
    sufficiently saturated o-minimal expansion $M$ of a field.
    Then $(\Goo(M);*)$ is bi-interpretable with the disjoint union of a 
    (possibly trivial) divisible torsion-free abelian group and finitely many 
    (possibly zero) real closed convexly valued fields.
  \end{theorem}
  \addtocounter{theorem}{-1}
}
\begin{proof}
  It follows from Lemma~\ref{l:oostruct}(ii) that $\Goo = (G^0)^{00}$ where $G^0$ is 
  the smallest definable subgroup of finite index, so we may assume $G=G^0$, 
  and hence (by \cite[Lemma~2.12]{P-groupsFieldsOMin}) that $G$ is definably 
  connected.

  Let $G' := (G,G)$ be the derived subgroup of $G$, and let $(\Goo)' := 
  (\Goo,\Goo)$ be the derived subgroup of $\Goo$.

  By Fact~\ref{f:DCLit}(i), $G'$ and $Z(G)$ are definable and definably compact, and 
  so Lemma~\ref{l:oostruct} applies to them.
  
  Let $H := G'/Z(G')$. Let $H_i$ be as in Fact~\ref{f:DCLit}(iii), so $H = \prod_i 
  H_i$.

  \begin{claim} \label{c:deriv00}
    \begin{enumerate}[(i)]\item $(G')^{00} \cong  H^{00}$ as groups.
    \item $H^{00} = (\prod_i H_i)^{00}$ is the definable internal direct 
    product of the $H_i^{00}$.
    \end{enumerate}
  \end{claim}
  \begin{proof}
    \begin{enumerate}[(i)]\item
    $(G')^{00}$ is torsion-free by Lemma~\ref{l:oostruct}(iii), and $Z(G')$ is 
    finite by Fact~\ref{f:DCLit}(iii), thus $Z(G') \cap (G')^{00} = \{e\}$. So by 
    Lemma~\ref{l:oostruct}(i), the quotient map induces such an isomorphism.

    \item Given Lemma~\ref{l:oostruct}(iv), we need only show that $H_i^{00}$ is 
    $(\Hoo;*)$-definable.

    By Lemma~\ref{l:Coo}, $C_H(H_i^{00}) = C_H(H_i)$, and since each $H_i$ is 
    centreless we have $H_i = \bigcap_{j \neq  i} C_H(H_j)$, thus $H_i^{00} = 
    \bigcap_{j \neq  i} C_{\Hoo}(H_j^{00})$ (using Lemma~\ref{l:oostruct}(iv) again).
    \end{enumerate}
  \end{proof}

  \begin{claim} \label{c:GooProd}
    \begin{enumerate}[(i)]\item $\Goo = Z(G)^{00}(G')^{00}$.
    \item $(G')^{00} = (\Goo)'$.
    \item $Z(G)^{00} = Z(\Goo)$.
    \item $\Goo$ is the definable internal direct product of $Z(G)^{00}$ and 
    $(G')^{00}$.
    \end{enumerate}
  \end{claim}
  \begin{proof}
    \begin{enumerate}[(i)]\item
      Since $Z(G)^{00}$ is central and each of $Z(G)^{00}$ and $(G')^{00}$ is 
      divisible and torsion-free, also $L := Z(G)^{00}(G')^{00}$ is divisible 
      and torsion-free. Now $G = Z(G)G'$ by Fact~\ref{f:DCLit}(ii), so any coset of 
      $L$ can be written as $zaL = (zZ(G)^{00})(a(G')^{00})$ with $z \in Z(G)$ 
      and $a \in G'$, thus the index of $L$ in $G$ is bounded by the product of 
      the indices of $Z(G)^{00}$ in $Z(G)$ and of $(G')^{00}$ in $G'$. So $L$ 
      is a bounded index subgroup. Thus we conclude by Lemma~\ref{l:oostruct}(iii).

    \item
      By Fact~\ref{f:DCLit}(iv) and \cite{dAndreaMaffei}, $H_i^{00} = 
      (H_i^{00},H_i^{00})_1$ for each $i$.
      So by Claim~\ref{c:deriv00}, $(G')^{00} = ((G')^{00},(G')^{00})_1$.
      Also $\Goo \cap G' = (G')^{00}$ by Lemma~\ref{l:oostruct}(ii),
      and so $(\Goo,\Goo)_1 = (G')^{00}$, and then since this is a subgroup we 
      also have $(\Goo)' = (\Goo,\Goo)_1$.

    \item
      By (i) it suffices to see that $Z((G')^{00}) = \{e\}$. But indeed, as in 
      Lemma~\ref{l:Coo}, $Z((G')^{00}) \leq  C_{G'}((G')^{00} = C_{G'}(G') = Z(G') = 
      \{e\}$. Alternatively, one can see this way that each $Z(H_i^{00}) = 
      \{e\}$, and apply Claim~\ref{c:deriv00}(ii).

    \item
      By Fact~\ref{f:DCLit}(ii), $Z(G)^0 \cap G'$ is finite. Hence also $Z(G)^{00} 
      \cap (G')^{00}$ is finite,
      and thus by Lemma~\ref{l:oostruct}(iii) it is trivial.
      Combining this with the previous items of this Claim, we conclude.
    \end{enumerate}
  \end{proof}

  Now each $H_i^{00}$ is bi-interpretable with a model of RCVF by 
  Fact~\ref{f:DCLit}(iv) and Theorem~\ref{t:mainBiterp}, and $Z(G)^{00}$ is (by 
  Lemma~\ref{l:oostruct}(iii)) a divisible torsion free abelian group, so we 
  conclude by Claim~\ref{c:GooProd}(iv), Claim~\ref{c:deriv00}, and Lemma~\ref{l:biterpProd}.
\end{proof}

\begin{remark}
  Since $M$ is an o-minimal expansion of a field, any $M$-definable real 
  closed field is $M$-definably isomorphic to $M$ as a field. Thus the valued 
  fields $R_i$ interpreted in the groups $H_i^{00}$ in the above proof are 
  $M$-definably isomorphic as fields.
  However, the disjoint union structure clearly does not define any such 
  isomorphisms between the $R_i$, and hence nor does the group $(\Goo;*)$.
\end{remark}

\bibliography{G00R}
\end{document}